\documentclass[12pt]{amsart}

\usepackage {amsfonts, amssymb, amscd, amsthm, amsmath, xypic}
\usepackage[dvips]{graphicx}

\newtheorem{Thm}{Theorem}[section]
\newtheorem{theorem}[Thm]{Theorem}
\newtheorem{lemma}[Thm]{Lemma}

\newtheorem{corollary}[Thm]{Corollary}

\newcommand{\vep}{\varepsilon}
\renewcommand{\phi}{\varphi}

\newcommand{\D}{\mathbb{D}}
\newcommand{\T}{\mathbb{T}}
\newcommand{\tz}{\mathbb{T}}

\newcommand{\rea}{{\rm Re}\,}

\title{On hypercyclic rank one perturbations of unitary operators}
\author{Anton Baranov, Vladimir Kapustin, Andrei Lishanskii}

\begin{document}
\sloppy
%\baselinestretched

%\large
%\normalsize

\address{ 
 Anton Baranov, 
\newline
Department of Mathematics and Mechanics,
St. Petersburg State University, 
\newline
St. Petersburg, Russia,
%\newline
%\phantom{x}\,\, and
%\newline
%National Research University  Higher School of Economics,
%\newline
%St. Petersburg, Russia,
\newline {\tt anton.d.baranov@gmail.com}
\newline\newline \phantom{x}\,\, Vladimir Kapustin,
\newline 
St. Petersburg Department of V.A. Steklov Mathematical Institute
\newline 
St. Petersburg, Russia,
\newline {\tt kapustin@pdmi.ras.ru}
\newline\newline \phantom{x}\,\, Andrei Lishanskii,
\newline 
Chebyshev Laboratory,
St. Petersburg State University,
\newline 
St. Petersburg, Russia,
\newline {\tt lishanskiyaa@gmail.com}
\newline\newline \phantom{x}
}
\thanks{The results of Section 2
were obtained with the support of the RFBR grant 16-01-00635. 
The results of Section 3 were 
obtained with the support of the Russian Science Foundation grant 
14-21-00035.}

\begin{abstract}
Recently, S. Grivaux showed that there exists 
a rank one perturbation 
of a unitary operator in a Hilbert space which is hypercyclic.
Another construction was suggested later by the first and the third 
authors. Here,
using a functional model for rank one perturbations of singular unitary 
operators, we give yet another construction of hypercyclic 
rank one perturbation of a unitary operator. In particular, we show that
any Carleson set on the circle can be the spectrum of a perturbed 
(hypercyclic) operator.

\end{abstract}


\maketitle

\section{introduction}

A continuous linear operator $T$ in a Banach (or Fr\'echet) space $F$
is said to be {\it hypercyclic} if there exists a vector
$f\in F$ such that its orbit $\{T^n f\}_{n=0}^\infty$
is dense in $F$. In this case the vector $f$ is said
to be {\it hypercyclic for} $T$. Though the property to have a vector with 
a dense orbit may look rather exotic and ``pathological'', it turns out that many
natural operators on spaces of analytic functions possess it (e.g., 
some classes of Toeplitz and composition operators, pseudodifferential operators 
on the space of all entire functions, weighted shifts, etc.). 
The theory of hypercyclic operators and, 
more generally, dynamics of linear operators, were an active
fields of research during the last three decades. For a comprehensive
account of the theory we refer to the recent monographs \cite{Bay-Mat, ge-per}.

One of the natural questions about hypercyclicity is its stability 
(or instability) with respect 
to small (in some sense) perturbations. In 1991 K.C.~Chan and J.H.~Shapiro \cite{chsh}
showed that there exist hypercyclic operators in a Hilbert space
of the form $I+K$, where $I$ is the identity operator and the compact operator 
$K$ belongs to any Schatten class. It is clear that $I+R$ can not be hypercyclic when
$R$ is a finite rank operator (since in this case $I+R^*$
has an eigenvalue). But if we replace $I$ by a unitary operator, 
hypercyclicity is possible. In 2010 S. Shkarin \cite{shk} produced an example of a unitary operator
$U$ such that $U+R$ is hypercyclic for some rank two operator $R$. Shkarin then asked
whether  $R$ can be taken to be of rank one. A positive answer was soon given by 
S. Grivaux \cite{gr1}:

\begin{theorem}[S. Grivaux, \cite{gr1}]
\label{main1}
There exists a unitary operator $U$ in the space $\ell^2$ 
and a rank one operator $R$ such that $U+R$ is hypercyclic.
\end{theorem}

In 2015 A.~Baranov and A.~Lishanskii \cite{BarL} gave another proof of this 
theorem using a functional model of rank one perturbations of unitary operators due to 
V.~Kapustin \cite{kap} and A.D.~Baranov and D.V.~Yakubovich \cite{bar-yak}. 
The constructions of \cite{gr1}
and \cite{BarL} were both based on the following interesting theorem 
(also due to S. Grivaux) which gives a condition sufficient for hypercyclicity 
in the case when the operator has sufficiently many 
eigenvectors corresponding to unimodular eigenvalues
with a  certain ``continuity property''.

\begin{theorem}[S. Grivaux, \cite{gr2}]
\label{main2}
Let $X$ be a complex separable infinite-dimensional Banach space, and
let $T$ be a bounded operator on $X$. Suppose that there exists a sequence 
$\{u_n\}_{n\ge 1}$ of vectors in $X$ having the following properties\textup:
\smallskip

{\rm (i)} $u_n$ is an eigenvector of $T$ associated to an eigenvalue $\lambda_n$ 
of $T$, with $|\lambda_n| = 1$ and $\lambda_n$ are all distinct\textup;
\smallskip

{\rm (ii)} ${\rm span} \{u_n: n\ge 1\}$ is dense in $X$\textup;
\smallskip
 
{\rm (iii)} for any $n \ge 1$ and any $\vep>0$, there exists 
$m\ne n$ such that $\|u_n -u_m\|<\vep$.
\smallskip
\\
Then $T$ is hypercyclic. 
\end{theorem}

Both in \cite{gr1} and \cite{BarL} the main technical 
difficulty was to prove the continuity property
(iii). In this note we give yet another construction of hypercyclic 
rank one perturbation of a unitary operator. 
This construction is related to a class of functional spaces studied 
in \cite{kap, kap2, mak}. One of the advantages of this approach 
is that the continuity property of eigenvectors
becomes almost automatic. Another (and apparently more important)
novelty is that now we have a lot of freedom in choosing the spectrum
of the perturbed (hypercyclic) operator (see Theorem \ref{carl1} below). 
\bigskip
%%%%%%%%%%%%%%%%%%%%%%%%%%%%%%%%%%%%%%%%%%%%%%%%%%%%%%%%%%%

\section{Carleson sets and spaces of analytic functions}

Let $\D =\{z\in\mathbb{C}: |z|<1\}$ be the unit disk and 
$\T =\{z\in\mathbb{C}: |z|=1\}$ be the unit circle.  
We denote by $H^2$ the standard Hardy space in $\mathbb{D}$ 
(identifying it with subspace of $L^2(\T)$ via its boundary values)
and we put $H^2_- = L^2(\T) \ominus H^2 = \overline{z H^2}$.
We denote by $S$ and $S^*$ the shift and backward shift operators 
(considered on the Hardy space or on the related spaces introduced below), 
$$
Sf(z)  = zf(z), \qquad S^*f(z) = \frac{f(z) - f(0)}{z}.
$$

Let $E$ be a closed subset of $\T$ with zero Lebesgue
measure and let $\{I_j\}$ be at most countable set of disjoint open arcs
$I_j \subset \T$ such that $\T\setminus E= \bigcup_j I_j$. The set $E$
is said to be a {\it Carleson set} (or a {\it set of finite entropy}) if 
$$
\sum_j |I_j| \log \frac{1}{|I_j|} <\infty
$$
(by $|I|$ we always denote the normalized Lebesgue measure of $I\subset \T$).
By the classical results of L.~Carleson \cite{carl}, sets of finite entropy 
are precisely those sets that may serve as zero sets of smooth 
up to the boundary) analytic functions in $\D$. Namely, a nonzero
function $f \in \mathit{Hol}(\D) \cap C^1(\overline{\D})$ 
(or $f \in \mathit{Hol}(\D) \cap C^\infty(\overline{\D})$)
vanishing on $E\subset \T$ exists if and only if $E$ is a Carleson set.

Recall that a set $E$ is said to be {\it perfect} if it is closed and has no 
isolated points. One of the main results of this note is the following theorem:

\begin{theorem}
\label{carl1}
For any perfect Carleson set $E$ on the unit circle $\mathbb{T}$
there exists a hypercyclic operator of the form $U+R$, where 
$U$ is unitary and $R$ of rank one, such that
$$
\sigma(U+R) = \sigma_p(U+R) = E.
$$
\end{theorem}

Here $\sigma(T)$ and $\sigma_p(T)$ denote the spectrum and 
the point spectrum of the operator $T$, respectively.

Appearance of Carleson sets in this context is not random. 
It is interesting to compare Theorem \ref{carl1} with the following remarkable 
result of N.G.~Makarov \cite{mak}: a set $F\subset \mathbb{T}$ is a point spectrum of some 
operator of the form $U+K$ where $U$ is unitary and $K$ is of trace class if and only if
$F$ is a countable union of Carleson sets. We do not know for which $F$ 
it is possible to find a hypercyclic operator $U+R$ as above with point spectrum $F$
(and, in particular, whether there exist countable unions of Carleson sets 
for which this is not true). 

Another interesting problem is to describe those unitary operators $U$ for which there 
exists a hypercyclic rank one perturbation $U+R$. This seems to be a difficult problem.
Our results imply that the essential spectrum of such operator $U$ can be any 
perfect Carleson set.
\medskip

\subsection{Construction of the space $H_*(\phi, \psi)$.}
The following class of spaces of analytic functions 
was introduced in \cite{kap} in connection with a model for rank 
one perturbations of unitary operators. 
Let $\phi$ and $\psi$ be the functions in $H^2 (\T)$ with the properties
$$
|\phi| = |\psi|\quad \text{a.e. on}\ \T \qquad\text{and} \qquad 
\psi(0) \neq 0.
$$
Let $H_*=H_*(\phi, \psi)$ be the space of functions defined by 
$$
H_* = \frac{H^2}{\phi} \cap \frac{H^2_-}{\bar \psi} = 
\{f: f\phi \in H^2, f\bar\psi \in H^2_-\}
$$
with the norm
$$
\|f\|_{H_*} := \|f \phi\|_{H^2} = \|f \bar\psi\|_{H^2_-}.
$$
It is clear that the space $H_*$ is invariant with respect to $S^*$. It is also shown 
in \cite{kap} that the inclusion $f\in H^2_-/\overline{\psi}$ and the fact 
that $\psi(0)  \ne 0$ imply that we can consider the values of $f\in H_*$ 
outside the unit disk and that there exists a finite limit 
$$
(zf)_\infty  = \lim_{z\to\infty} zf(z).
$$                                  

Assume now that $E\subset\mathbb{T}$
is a perfect Carleson set. 
Let $\phi \in C^1(\overline{\D})$ be such that $\phi(0) \ne 0$ and  
$\{z\in\mathbb{T}:\, \phi(z) = 0\} = E$. Then it is clear that
any function of the form $(z-\lambda)^{-1}$, $\lambda\in E$, belongs to the space 
$H_* = H_*(\phi, \phi)$ associated with the pair $(\phi, \phi)$. Indeed, 
$$
\frac{\phi}{z-\lambda} \in H^2, \qquad 
\frac{\overline{\phi}}{z-\lambda} = \frac{1}{\lambda}
\overline{\bigg(\frac{z\phi}{\lambda-z}\bigg)} \in H^2_-.
$$

We will use the fact that the functions in the space $H_*$ have analytic continuation
through $\mathbb{T} \setminus E$.

\begin{lemma}
\label{lem1}
Let $\mu\in \mathbb{T} \setminus E$. Then any function in $H_*$ 
admits analytic continuation into a neighborhood of $\mu$ 
and the mapping $f\mapsto f(\mu)$ 
is a bounded functional on $H_*$. 
\end{lemma}

\begin{proof}
By definition, $f\in H_*$ if and only if $\phi f\in H^2$ and 
$\phi \bar z\bar f\in H^2$. The last condition can be rewritten 
as follows: there exists a function $g$ meromorphic and of Nevanlinna class 
in $\{|z|>1\}$ such that $g=f$ a.e. on $\mathbb{T}$ and 
$$
\frac{1}{z} 
\overline{g\Big(\frac{1}{\bar z}\Big)}\phi(z) \in H^2.
$$ 

If $\mu \in \mathbb{T} \setminus E$, then there exist an open neighborhood $V$
of $\mu$ and $\delta>0$ such that $|\phi(z)|\ge \delta$, 
$z\in \overline{\mathbb{D}} \cap V$. It follows that, for some $\vep>0$,
$$
\sup_{1< r<1+\vep} \int_{\{|e^{i\theta} - \mu|<\vep\}} |g(re^{i\theta})|^2 
d\theta <\infty.
$$
Hence, $g$ is a usual analytic continuation of $f$ through the arc 
$\{|e^{i\theta} - \mu|<\vep\}$. Moreover,
$$
|f(\mu)|^2 \le \frac{1}{\pi \vep^2} 
\int_{\{|z-\mu|<\vep\}} |f(z)|^2 dm_2(z) \le  \frac{1}{\pi \vep^2 \delta^2}  
\|\phi f\|_{H^2}^2.
$$
Here $m_2$ is the planar Lebesgue measure.
\end{proof}
\bigskip

%%%%%%%%%%%%%%%%%%%%%%%%%%%%%%%%%%%%%%%%%%%%%%%%%%%%%%%%%%%

\section{Construction of a hypercyclic rank one perturbation.}

In this section we use the spaces of analytic functions 
from the previous section to construct a rank one perturbatrion of a 
unitary operator which satisfies the conditions of Grivaux' theorem (Theorem \ref{main2}).

As above let $E\subset\mathbb{T}$ be a perfect Carleson set and  
let $\phi \in C^1(\overline{\D})$ be such that $\phi(0) \ne 0$ and  
$\{z\in\mathbb{T}:\, \phi(z) = 0\} = E$. Consider the space 
$H_*$ associated with the pair $(\phi, \phi)$. 
Then any function of the form $(z-\lambda)^{-1} \in H_*$ for 
$\lambda\in E$. It follows from the inclusion 
$\phi \in C^1(\overline{\D})$ that 
$$
\sup_{\lambda\in E, z\in\T}\Big|\frac{\phi(z)}{z-\lambda}\Big| <\infty,
$$ 
whence, by the Lebesgue dominated convergence theorem, 
\begin{equation}
\label{por1}
\Big\|\frac{\phi}{z-\mu} - \frac{\phi}{z-\lambda} \Big\|_{H^2}  \to 0 \quad\text{as}
\ \mu\to \lambda, \ \mu\in E.
\end{equation} 

Now, let $\{\lambda_j\}_{j\in J}$ be  a countable subset of $E$ without isolated points
and put 
$$
H_0 = \overline{\text{Span}}_{H_*} \bigg\{\frac{1}{z - \lambda_j}:\, j\in J\bigg\}.
$$

\begin{theorem}
\label{new}
The operator $S^*|_{H_0}$ is a hypercyclic operator which is a 
rank one perturbation of some unitary operator. Moreover, 
$\sigma(S^*|_{H_0}) = \sigma_p(S^*|_{H_0}) = \overline{E}$
where $\overline{E} = \{\bar \lambda:\, \lambda\in E\}$.
\end{theorem}

Put $H_1 = \{f\in H_0: f(0) = 0\}$ and 
$\widetilde{H}_1 = \{f\in H_0:\, (zf)_\infty = 0\}$. Then it is clear that 
$S^*$ acts isometrically on $H_1$. We have the following lemma:

\begin{lemma}
\label{lem2}
$S^* H_1 = \widetilde{H}_1$.
\end{lemma}

\begin{proof}
If $g = \sum\limits_j \frac{c_j}{z - \lambda_j} \in H_1$ (we consider only finite sums), 
then $g(0) = -\sum\limits_j \frac{c_j}{\lambda_j} = 0$.
In this case $S^*g = \sum\limits_j \frac{c_j}{\lambda_j (z - \lambda_j)}$, 
so $(zS^*g)_\infty = 0$.
Since $(zf)_\infty$ is a continuous linear functional 
on $H_*$, we conclude that $(zf)_\infty = 0$ whenever $f\in S^* H_1$. Thus, 
$S^*H_1  \subset \widetilde{H}_1$.
\medskip

To prove the converse inclusion, let us note first that finite sums of the form 
$\sum\frac{c_j}{z-\lambda_j}$ are dense in $\widetilde{H}_1$.
Let $f\in H_0$ and $(zf)_\infty = 0$.
For any $\vep>0$ we can choose a finite linear combination 
$\sum\frac{c_j}{z-\lambda_j}$ such that $\Big\|f- \sum\frac{c_j}{z-\lambda_j}\Big\|<\vep$.
Since $h\mapsto (zh)_\infty$ is a bounded functional on $H_0$, there
exists a constant $A>0$ (independent of $f$) such that 
$$
\Big| \sum c_j\Big| = \bigg|\bigg(z\Big(\sum\frac{c_j}{z-\lambda_j}-f\Big)\bigg)_\infty\bigg| 
<A\vep.
$$
Let $\lambda_0\in E$ be fixed which is independent on $f$. 
Then it is easy to see that 
$$
\tilde f(z) =  \sum\frac{c_j}{z-\lambda_j} - \frac{\sum_j c_j}{z-\lambda_0} \in 
\widetilde{H}_1
$$
and $\|f-\tilde f\|\le \vep(1+A \|(z-\lambda_0)^{-1}\|_{H^*})$.

Now let $f(z) =  \sum\frac{c_j}{z-\lambda_j}$ (a finite sum) and $(zf)_\infty = \sum c_j 
=0$. Put $g= \sum \frac{\lambda_j c_j}{z-\lambda_j}$. Then $g(0) = 0$ (i.e., $g\in H_1$)
and 
$$
S^* g = \frac{1}{z} \bigg(\sum \frac{\lambda_j c_j}{z-\lambda_j} + \sum c_j \bigg) = f.
$$
Thus, $\widetilde{H}_1 \subset S^*H_1$.
\end{proof}
\medskip

\subsection{Proof of Theorem \ref{new}.} 
First we verify that the operator $S^*|_{H_0}$ satisfies the conditions of 
Theorem \ref{main2}. It is clear that the functions $f_j = \frac{1}{z - \lambda_j}$ 
are eigenvectors of the operator $S^*|_{H_0}$ with eigenvalues $\frac{1}{\lambda_j}$. 
Indeed, 
$$
S^*f_j (z) = 
\frac{1}{z} \Big(\frac{1}{z - \lambda_j} + \frac{1}{\lambda_j}\Big) = 
\frac{1}{\lambda_j (z - \lambda_j)}.
$$
By the definition of the space $H_0$ the set $\{f_j\}_{j\in J}$ 
is complete in $H_0$. Since the set 
$\{\lambda_j\}_{j\in J}$ has no isolated points, the continuity property (iii)
in Theorem \ref{main2} follows from \eqref{por1}.
Thus, $S^*|_{H_0}$ is hypercyclic.
\medskip

Next we need to show that $S^*|_{H_0}$  is a rank one perturbation of some
unitary operator. Let $H_1$ and $\widetilde{H}_1$ be as above. 
It is clear that 
$H_0 = H_1 \oplus H_2$, where $H_2 = \text{Lin}\{g\}$ for some $g\in H_0$ such that 
$g(0) =~1$. By Lemma \ref{lem2} $S^* H_1 = \widetilde{H}_1$ and so 
the closed subspace $S^*H_1$ has codimension
1 in $H_0$. Then we can define a unitary operator $U$ as follows: let 
$U|_{H_1} = S^*|_{H_1}$ and let $U|_{H_2}$ be an arbitrary unitary operator
of one-dimensional subspace $H_2$ onto one-dimensional subspace $H_0\ominus{S^*H_1}$.
Clearly, $S^*|_{H_0}$ will be a rank one perturbation of $U$. 
\medskip

It remains to show that $\sigma(S^*|_{H_0}) = \sigma_p(S^*|_{H_0}) = \overline{E}$. 
It is clear that $\overline{E} \subset \sigma_p(S^*|_{H_0})$. 
We show that for $\lambda \notin \overline{E}$ the equation
$S^*f- \lambda f = g$ has a unique solution $f\in H_0$ for any $g\in H_0$.  
Indeed, this equation is equivalent to 
$$
f = \frac{zg +f(0)}{1-\lambda z}. 
$$
If $\mu = \lambda^{-1} \in\mathbb{D}$ or $\mu \in \mathbb{T}\setminus E$, then
by lemma \ref{lem1} 
any function $g\in H_0$ is analytic in the neighborhood of $\mu$ and so there exists 
a unique $c\in\mathbb{C}$, namely $c=\mu g(\mu)$, such that 
$f = \frac{zg-c}{\lambda(z-\mu)}$ belongs to $\phi^{-1} H^2$. The inclusion 
$\phi \bar z \bar f \in H^2$ follows immediately from the inclusion 
$\phi \bar z \bar g \in H^2$. If $\lambda \in \mathbb{D}$, then $\phi f\in H^2$
for any $f$ of the form $f = \frac{zg-c}{\lambda(z-\mu)}$. Also,
$$
\phi \bar z \bar f = 
\phi \bar z
\frac{\bar z \bar g - \bar c}{1 - \bar \lambda \bar z}
= \frac{\phi \bar z \bar g - \bar c \phi}{z - \bar \lambda}  \in H^2
$$
if and only if $\bar c = (\phi \bar z \bar g)(\bar \lambda)/\phi(\bar\lambda)$. Thus,
we have seen that there is a unique solution $f\in H_*$ for any $g\in H_*$. 
If $g$ is a finite linear combination of $(z-\lambda_j)^{-1}$, then it is easy 
to see that $f$ also is of this form. Hence, 
$f\in H_0$ whenever $g\in H_0$. This proves Theorem \ref{new}. 
\qed
\medskip

\subsection{A link to nearly invariant subspaces and Clark measures}
In this subsection we discuss the relation of our construction with Hitt--Sarason theory
of nearly invariant subspaces. A closed subspace $X \subset H^2$ is said to be 
\textit{nearly invariant}, if for any $f \in X$ such that $f(0)=0$ we have 
$S^*f \in X$. Recall that by the Beurling theorem any $S^*$-invariant subspace
is of the form $K_\Theta = H^2\ominus\Theta H^2$ for some inner function $\Theta$ 
(see, e.g., \cite{nik1}). The structure of nearly invariant subspaces was described 
by D.~Hitt \cite{hitt} and D.~Sarason \cite{sar2}. 

\begin{theorem}[Hitt, Sarason, \cite{hitt, sar2}]
\label{hit}
If $X$ is nearly invariant, then $X = g K_\Theta$, where $\Theta$ is an 
inner function, $g$ is outer, and the mapping $M_g: f \mapsto gf$ is an isometry on $K_\Theta$.
In particular, $M_g$ maps $K_\Theta$ isometrically onto $X$.
\end{theorem}

Let us show that
$$
\varphi H_0 = \varphi \overline{\text{Span}}_{H_*} 
\bigg\{\frac{1}{z - \lambda}:\, \lambda \in E \bigg\}.
$$
is a nearly-invariant subspace of $H^2$. Indeed, 
for finite sums $\varphi\sum \frac{c_j}{z - \lambda_j}$ 
(which are dense in $\varphi H_0$) we have 
$$
\sum \frac{c_j}{z - \lambda_j}\bigg|_{z=0} = -\frac{c_j}{\lambda_j} = 0,
$$ 
and so 
$$
\frac{1}{z} \bigg(\varphi \sum \frac{c_j}{z - \lambda_j}\bigg) = 
\frac{1}{z} \bigg(\varphi \sum \Big(\frac{c_j}{z - \lambda_j} + \frac{c_j}{\lambda_j}\Big)
\bigg) = \varphi \bigg(\sum \frac{c_j}{\lambda_j (z - \lambda_j)}\bigg) \in \varphi H_0.
$$
Therefore, by the Hitt--Sarason theorem $\varphi H_0 = g K_\Theta$ for some $g$ 
and $\Theta$.

If $f \in H_0$, then $\frac{\varphi f}{g} \in K_\theta$ and
$$
\|f\|_{H_0} = \|\varphi f\|_{H^2} = \Big\|\frac{\varphi f}{g}\Big\|_{H^2}
$$
and so the mapping $M_{\phi/g}: 
f\mapsto \phi f/g$ is a unitary operator from $H_0$ onto $K_\Theta$.
It is clear that operators $S^*|_{H_1}$ (where $H_1 = \{f\in H_0: f(0) = 0\}$)
and $S^*|_{\{f\in K_\Theta:\, f(0)=0\}}$ are unitarily equivalent. 
Since $S^*|_{H_1}$ coincides on $H_1$ with the unitary operator $U$
constructed in Theorem \ref{new}, we conclude that $\tilde U = M_{\phi/g}^{-1}
U M_{\phi/g}$ is a unitary operator on $K_\Theta$ which is a rank one perturbation of 
$S^*|_{K_\Theta}$. 

Unitary rank one perturbations of $S^*|_{K_\Theta}$ were described by D.N. Clark 
in a seminal paper \cite{cl}. This description involves a construction of a 
special family of measures. Let $\alpha\in\mathbb{T}$. Then the function 
$\frac{\alpha+\theta}{\alpha-\theta}$ has positive real part in $\mathbb{D}$, and so
there exists a finite positive measure
$\mu_\alpha$ on $\tz$ (singular with respect to the Lebesgue measure) such that
$$
\rea \frac{\alpha+\theta(z)}{\alpha-\theta(z)}=
\frac{1}{\pi} \int_\mathbb{T} \frac{1-|z|^2}{|\tau - z|^2}\,
d \mu_\alpha(\tau), \qquad z\in\mathbb{D}.
$$
Clark showed that the set of unitary rank one perturbations of $S^*|_{K_\Theta}$
can be parametrized by $\alpha\in\mathbb{T}$ and their spectral measures
are exactly the measures $\mu_\alpha$. 

\begin{corollary}
Let $E\subset \mathbb{T}$ be a perfect Carleson set 
and let $\phi$ be the corresponding smooth function with boundary 
zero set $E$. Then the space $\phi H_0$ is a nearly invariant subspace. 
Moreover, if $\Theta$ is the associated inner function from the Hitt--Sarason theorem. 
then the spectral measure of the unitary operator $U$ which has a hypercyclic rank one 
perturbation is the Clark measure $\mu_\alpha$ for $\Theta$. 
\end{corollary}

It would be interesting to use this relation to characterize 
(at least a subclass of) unitary operators $U$ which have a hypercyclic rank
one perturbation. The difficulty here is that the proof of the Hitt--Sarason theorem
is nonconstructive and, even if we have explicit knowledge of 
$\phi$ and $H_0$, it is difficult to translate it into the information
about $\Theta$. In particular, it would be interesting to know whether
the corresponding Clark measure $\mu_\alpha$ can have a nontrivial 
singular continuous (i.e., without atoms) part.

\end{document}